\documentclass[reqno,12pt]{amsart}
\usepackage{amssymb,amsmath,amsthm,amsxtra,mathrsfs,calc,bm, color}
\usepackage{enumerate}
\usepackage[normalem]{ulem}
\usepackage[margin=1in]{geometry}
\usepackage{mathtools}
\usepackage{mleftright}
\mleftright

\usepackage{nccmath}
\usepackage{cases}
\usepackage{calc}
\usepackage{graphicx}
\usepackage{hyperref}

\def\Z{\mathbb{Z}}
\def\Q{\mathbb{Q}}
\def\R{\mathbb{R}}
\def\H{\mathcal{H}}

\def\C{\mathbb{C}}

\DeclareMathOperator{\im}{Im}
\DeclareMathOperator{\re}{Re}

\def\SL{{\rm SL}}
\def\GL{{\rm GL}}
\newcommand{\pfrac}[2]{\left(\frac{#1}{#2}\right)}

\newcommand{\pMatrix}[4]{\left(\begin{matrix}#1 & #2 \\ #3 & #4\end{matrix}\right)}

\renewcommand{\pmatrix}[4]{\left(\begin{smallmatrix}#1 & #2 \\ #3 & #4\end{smallmatrix}\right)}
\renewcommand{\bar}[1]{\overline{#1}}

\renewcommand{\tilde}{\widetilde}

\renewcommand{\a}{\mathfrak{a}}
\renewcommand{\b}{\mathfrak{b}}
\renewcommand{\c}{\mathfrak{c}}

\newtheorem{theorem}{Theorem}[section]
\newtheorem{lemma}[theorem]{Lemma}

\newtheorem{proposition}[theorem]{Proposition}

\theoremstyle{remark}

\numberwithin{equation}{section}

\mathtoolsset{showonlyrefs}

\newcommand{\sums}{\sideset{}{^*}\sum}

\newcommand{\M}{\mathcal M}
\newcommand{\scN}{\mathcal N}
\newcommand{\K}{\mathcal K}
\newcommand{\scR}{\mathcal R}
\newcommand{\scT}{\mathcal T}

\renewcommand{\S}{\mathcal S}

\renewcommand{\Re}{\operatorname{Re}}
\newcommand{\inv}{^{-1}}

\newmuskip\pFqskip
\mathchardef\pFcomma=\mathcode`,

\title[Rankin-Selberg level reciprocity]{Level Reciprocity in the twisted second moment of Rankin-Selberg $L$-functions}

\author{Nickolas Andersen}
\address{UCLA Mathematics Department, Los Angeles, CA 90095}
\email{nandersen@math.ucla.edu}

\author{Eren Mehmet Kiral}
\address{Wako-Shi, Saitama, Japan}
\email{erenmehmetkiral@protonmail.com}

\thanks{This material is based upon work supported by the National Science Foundation under Grant No.~1440140, while the authors were in residence at the Mathematical Sciences Research Institute in Berkeley, California, during the Spring semester of 2017.
The first author is also supported by NSF grant DMS-1701638.}

\date{\today}

\begin{document}

\begin{abstract}
We prove an exact formula for the second moment of Rankin-Selberg $L$-functions $L(\frac 12,f\times g)$ twisted by $\lambda_f(p)$, where $g$ is a fixed holomorphic cusp form and $f$ is summed over automorphic forms of a given level $q$.
The formula is a reciprocity relation that exchanges the twist parameter $p$ and the level $q$. The method involves the Bruggeman/Kuznetsov trace formula on both ends; finally the reciprocity relation is established by an identity of sums of Kloosterman sums.
\end{abstract}

\maketitle

\section{Introduction}

Let $p,q$ be distinct primes. 
Fix a holomorphic cusp form $g$ of weight $\kappa$ on $\SL_2(\Z)$. 
In this paper we establish a reciprocity relation between the twisted second moment of the central values of Rankin-Selberg $L$-functions:
\[
	\sum_{f \text{ level } q} \omega_f \lambda_f(p) L(\tfrac{1}{2},f\times g)^2 \leadsto \sum_{f \text{ level } p} \omega_f \lambda_f(q) L(\tfrac{1}{2}, f\times g)^2.
\]
Here $\omega_f = \pm 1$ is the eigenvalue of $f$ under the Fricke involution and $\lambda_f(p)$ is the $p$\textsuperscript{th} Hecke eigenvalue of $f$. The sum on each side should be understood as a complete sum/integral over the full spectrum of level $q$ or $p$ modular forms, including the holomorphic, discrete, and continuous spectra. The exact formulation is given below in Theorem \ref{thm:MainTheorem}. 

Our work is motivated by the case when $g$ is an Eisenstein series and the sum is over Hecke cusp forms $f$ of a given weight. In that case preliminary calculations with the Petersson trace formula and transforms on sums of Kloosterman sums lead to a formula of rough shape
\begin{equation}\label{eq:twisted-fourth-moment}
	\sum_{f \text{ level } q} \lambda_f(p) L(\tfrac12,f)^4 \leadsto \sum_{f \text{ level } p} \lambda_f(q) L(\tfrac12,f)^4.
\end{equation}
Note that when $f$ is a holomorphic modular form, $\omega_f = -1$ implies $L(\frac 12,f) = 0$, and therefore the identity \eqref{eq:twisted-fourth-moment} does not include $\omega_f$.
One may use the amplification method in conjunction with such an identity to obtain a subconvexity result for $L(\frac12, f)$ in the level aspect. 

We were led to consider such a reciprocity relation after the works of  Conrey~\cite{conrey2007mean}, Young~\cite{YoungReciprocity}, and Bettin~\cite{BettinReciprocity},
who discovered and elaborated upon an identity relating $\M(a,q)$ to $\M(-q,a)$, where $\M(a,q)$ is the  second moment of Dirichlet $L$-functions $L(\frac 12,\chi)$ modulo $q$, twisted by $\chi(a)$. 

Our method is structurally similar to Motohashi's proof of a beautiful formula discovered by Kuznetsov and then fully proven in \cite{MotohashiFourthMomentFE}. 
Similar to Motohashi we apply the Bruggeman/Kuznetsov trace formula, followed by the
 $\GL_2$-Voronoi fomula twice, giving us again a sum over Kloosterman sums. 
We then apply the Bruggeman/Kuznetsov trace formula again in order to obtain the reciprocal moment. 

Our work is distinct in at least three ways from that of Motohashi. 
First, we are working in the congruence subgroup $\Gamma_0(q)$ and twisting by the Fourier coefficient $\lambda_f(p)$.
This allows us to see the reciprocity relation exchanging the level and the twist parameters. 
Second, we twist our moments further by $\omega_f$, so that when we apply the Bruggeman/Kuznetsov trace formula we are working with the cusp-pair $0\infty$. 
The Kloosterman sums associated to the cusp-pair $0\infty$ feature $p$ and $q$ in more symmetric roles, and it becomes conceptually clear how the reciprocity occurs (see Theorem~\ref{thm:S-reciprocity} below). 
As can be seen in \cite{KiralYoungFifthMoment} and \cite{BlomerKhanReciprocity} the trick of moving to the $0\infty$ cusp-pair may be avoided in the fourth moment case by inserting an arithmetic reciprocity relation between the Voronoi formulas, but this trick does not work in the Rankin-Selberg second moment case (more on \cite{BlomerKhanReciprocity} below).
Third, our formula relates the twisted Rankin-Selberg second moment rather than the fourth moment. 
In principle, we could obtain the fourth moment if $g$ were chosen as an Eisenstein series. 
Practically, this corresponds to replacing every instance of $\lambda_g(m)$ in this paper by $\tau_w(m) = \sum_{ab = n}\left(\tfrac{a}{b}\right)^w$, which introduces {main terms} at various points.

Motohashi has produced other beautiful formulas relating different moments of $L$-functions. For example, in \cite{MotohashiFourthMomentZetaThirdMomentMaass} he gives an exact identity relating the weighted fourth moment of the Riemann zeta function on the critical line to third moments of central values of Maass forms of level $1$. 
Later with Ivic \cite{IvicMotohashiFourthMomentZetaError} they use this exact formula to give an asymptotic for the fourth moment of the Riemann zeta function with an error term of size $O(T^{2/3}(\log T)^c)$. 
Several authors have discovered and applied identites between moments of $L$-functions either exact or approximate; see \cite{PetrowTwistedMotohashi,YoungFourthMoment} and the references therein.

Recently Blomer, Li and Miller \cite{BLMSpectralReciprocity} announced an identity involving the first moment of $L(\frac 12,\Pi \times u_j)$ where $\Pi$ is a self-dual cusp form on $\GL_4$ and $u_j$ runs over $\GL_2$ Maass forms.
Notice that Motohashi's formula on the fourth moment could be interpreted as the case where $\Pi$ is a $4 = 1 + 1 + 1 + 1$ isobaric sum. 

During the preparation of this manuscript, Blomer and Khan posted their preprint \cite{BlomerKhanReciprocity}, in which they addressed the twisted fourth moment problem \eqref{eq:twisted-fourth-moment} and realized independently that one obtains a kind of reciprocity relation exchanging $p$ and $q$. 
We assume $p$ and $q$ to be prime for simplicity; {additionally, we include the Fricke eigenvalue in the moment, which allows us make use of arithmetic features coming from the $0\infty$ cusp-pair Kloosterman sums in the Kuznetsov formulas.} 
As an application, Blomer and Khan sum over the twist variable and reconstruct the subconvexity-implying fifth moment bound in \cite{KiralYoungFifthMoment}.
 They start by considering the moment $\sum_{\pi \text{ level } q}  L(1/2,F \times \pi) L(1/2,\pi) \lambda_\pi(\ell)$ where $F$ is a $\GL_3$ automorphic form. This is the $4 = 3 + 1$ decomposition {in the framework of \cite{BLMSpectralReciprocity}}.
 When $F$ is an Eisenstein series one obtains the twisted fourth moment.  

In the framework above, our result corresponds to the $4 = 2 + 2$ setup.
This difference is the reason why we use the $\GL_2$ instead of the $\GL_3$ Voronoi summation formula.

\section{Statement of Results} \label{sec:preliminaries}

We begin by fixing notation, which we mostly borrow from \cite{IwaniecSpectralBook}. 
Let $\Gamma = \Gamma_0(N)$ for some squarefree integer $N$ and let $k\geq 0$ be an even integer.
The weight $k$ Petersson inner product is defined as 
\[
  \langle h_1,h_2\rangle = \iint_{\Gamma\backslash \H} h_1(z) \overline{h_2(z)} y^k \, \frac{dx dy}{y^2}. 
\]
Here $z=x+iy$ and $h_1,h_2$ are holomorphic cusp forms of weight $k$ or Maass cusp forms (in the latter casse $k=0$).

Let $S_k(N)$ denote the space of holomorphic cusp forms of weight $k$ on $\Gamma_0(N)$, and let $\mathcal B_k(N)$ denote an orthonormal basis of $S_k(N)$.
We will always use $f$ or $g$ to denote an element of $\mathcal B_k(N)$.
The Fourier expansion of such an $f$ at a cusp $\a$ of $\Gamma$ is given by
\[
  j(\sigma_\a,z)^{-k} f(\sigma_\a z) = \sum_{n=1}^\infty \rho_{\a f}(n) e(nz),
\]
where $j(\pmatrix abcd, z) = cz+d$ and $\sigma_\a$ is a scaling matrix (see Section~\ref{sec:kloo-b/k}).
When $\a=\infty$ we will often drop the dependence on $\a$ from the notation.
We adopt the standard notation $e(x):= e^{2\pi ix}$.
We normalize the coefficients $\rho_{\a f}(n)$ by setting
\[
  \nu_{\a f}(n) = \left(\frac{\pi^{-k}\Gamma(k) }{(4n)^{k-1}}\right)^{\frac12} \rho_{\a f}(n).
\]
Without loss of generality we may assume that the elements of $\mathcal B_k(N)$ are eigenforms of the Hecke operators $T_n$ for $(n,N)=1$ with eigenvalues $\lambda_f(n)$, and that they satisfy
\begin{equation} \label{eq:fricke}
  N^{-\frac k2} z^{-k} f(-1/Nz) = \omega_f f(z)
\end{equation}
where $\omega_f=\pm 1$ is the eigenvalue of the Fricke involution.
Since the Fricke involution swaps the cusps $\infty$ and $0$, we have the relation
\begin{equation}
  \rho_{0 f}(n) = \omega_f \, \rho_{\infty f}(n).
\end{equation}

Similarly, let $\mathcal U(N) = \{u_j\}$ denote a complete orthonormal system of Maass cusp forms with Fourier expansions
\begin{equation}
  u_j(\sigma_\a z) = \sqrt y \sum_{n \neq 0} \rho_{\a j}(n) K_{i t_j}(2\pi |n| y) e(nx),
\end{equation}
where  $\frac 14+it_j$ is the Laplace eigenvalue and $K_\nu(x)$ is the $K$-Bessel function.
We normalize the coefficients $\rho_{\a j}(n)$ by
\begin{equation}
  \nu_{\a j}(n) = \left(\frac{\pi}{\cosh(\pi t_j)}\right)^{\frac12} \rho_{\a j}(n).
\end{equation}
We may also assume, as above, that the $u_j$ are eigenforms for the Hecke operators $T_n$ for $(n,N)=1$ and for the Fricke involution (i.e. that $u_j$ satisfies \eqref{eq:fricke} with $k=0$).
We write their eigenvalues $\lambda_j(n)$ and $\omega_j$, respectively.

Let $(\nu,\lambda)$ denote one of the pairs $(\nu_{f},\lambda_f)$ or $(\nu_{j},\lambda_j)$.
As long as $(n,N) = 1$ we have the relation
\begin{equation}\label{eq:hecke}
  \nu(m)\lambda(n) = \sum_{d | (m,n) } \nu\left(\frac{mn}{d^2}\right)
\end{equation}
see \cite[(8.37)]{IwaniecSpectralBook}. 
This implies implies that
\[
  \nu(n) = \nu(1)\lambda(n) 
\]
as long as $(n,N) = 1$. 
Both sides are zero when $(\nu,\lambda)$ does not correspond to a newform.

For each cusp $\c$ of $\Gamma$ (see Section~\ref{sec:kloo-b/k}) and for $\re(u)>1$ let
\[
  E_\c(z,u) := \sum_{\gamma \in \Gamma_\infty\backslash \Gamma} \im(\sigma_\c\inv \gamma z)^u
\]
denote the Eisenstein series associated to $\c$.
This has Fourier expansion
\begin{equation}\label{eq:EisensteinFourierExpansion}
  E_\c(\sigma_\a z,u) 
  = \delta_{\a\c} y^{u} + \rho_{\a\c}(0,u) y^{1 - u} 
    + \sqrt y \sum_{n \neq 0} \rho_{\a\c}(n,u) K_{u-\frac 12}(2\pi |n|y) e(nx)
\end{equation}
which has a meromorphic continuation to $u\in \C$.
On the line $\re(u) = \frac 12$ we normalize the coefficients by
\begin{equation}
  \nu_{\a\c}(n,t) = \left(\frac{\pi}{\cosh \pi t}\right)^{\frac12} \rho_{\a\c}  (n,\tfrac 12+it).
\end{equation}

For the remainder of the paper, fix a
 holomorphic newform $g\in S_\kappa(1)$.
For $s=\sigma+it$ with $\sigma$ sufficiently large,
let us call
\begin{equation}\label{eq:RawLFunction}
  \tilde L(s, h \times g) = \zeta_N(2s) \sum_{n=1}^\infty \frac{\lambda_g(n)\nu_h(n)}{n^s},
\end{equation}
where $\zeta_N(s) = \prod_{p \nmid N} (1 - p^{-s})\inv$.
This is the ``raw $L$-function'' involving {$\nu_h$ as opposed to $\lambda_h$}; it has an analytic continuation and a functional equation.
We choose this notation {in order to simultaneously cover oldforms.} 
The coefficients $\nu_h(n)$ come up in applications of the Bruggeman/Kuznetsov trace formula and hence the Dirichlet series $\tilde L(s, g \times h)$ naturally appears.
If $h$ is a newform (holomorphic or Maass), then {we} simply {have} $\tilde L(s,h) = \nu_h(1) L(s,h)$, where $L(s,h)$ is the usual $L$-function of $h$.

Let $\varphi$ be a smooth function defined on the nonnegative reals such that 
\begin{equation}\label{eq:phiKuznetsovdGrowth}
  \varphi(0) = 0 \quad \text{ and } \quad \varphi^{(j)}(x) \ll (1+x)^{-2-\epsilon} \quad \text{ for }j=0,1,2,
\end{equation}
and let $\varphi_h$ and $\varphi_+$ denote the integral transforms in \eqref{eq:KuznetsovTransformDefinitions} below.
Define
\begin{equation}\label{eq:Ngdisc}
  \scN_g^{d}(p,q;s;\varphi) = \sum_{u_j \in \mathcal U(q)} \omega_j \, \varphi_+(t_j) \tilde L(s, g \times u_j)^2\lambda_j(p),
\end{equation}
\begin{equation}\label{eq:Nghol}
  \scN_g^{h}(p,q;s;\varphi) = \sum_{\substack{\ell \text{ even}}} i^{\ell} \varphi_h(\ell) \sum_{f \in \mathcal B_\ell(q)}  \omega_f \, \tilde L(s,g \times f)^2 \lambda_f(p),
\end{equation}
and
\begin{equation}\label{eq:Ngcts}
  \scN_g^{c}(p,q;s;\varphi) = \frac{1}{4\pi}\sum_{\c} \int_{-\infty}^\infty \varphi_+(t) \tilde L\left(s, g\times E_\c(\sigma_0*, \tfrac12 + it)\right) \tilde L\left(s, g\times E_\c(*,\tfrac12 - it)\right) \tau_{it}(p) \, dt,
\end{equation}
where $\sum_\c$ is over a set of inequivalent cusps of $\Gamma$.
Call
\begin{equation}\label{eq:NgDefinition}
  \scN_g(p,q;s;\varphi) = \scN_g^d(p,q;s;\varphi) + \scN_g^c(p,q;s;\varphi) + \scN_g^h(p,q;s;\varphi).
\end{equation}
With $\varphi$ as above, the function $\scN_g(p,q;s;\varphi)$ is holomorphic at $s=\frac 12$, which is the point we are most interested in, and we let $\scN_g(p,q;\varphi) := \scN_g(p,q;\frac 12;\varphi)$.

We are now ready to state our main theorem.
\begin{theorem}\label{thm:MainTheorem}
Let $p$ and $q$ be distinct primes and let $\phi$ be a function on $[0,\infty)$ satisfying the conditions in the beginning of Theorem \ref{thm:S-reciprocity}.
With the notation above, define
\begin{equation}
	\M_g(p,q;\varphi) := (1-p^{-2}) \, \scN_g(p,q;\varphi) - \frac{2\lambda_g(p)}{\sqrt p}(1-p^{-1}) \, \scN_g(1,q;\varphi) + \frac{1}{\sqrt p} \, \scN_g(1,pq;\varphi).
\end{equation}
Then
\begin{equation}
	\sqrt q \, \M_g(p,q;\phi) = \sqrt p \, \M_g(q,p;\Phi),
\end{equation}
where $\Phi$ is an integral transform of $\phi$ given in \eqref{eq:Phi-def}.
\end{theorem}

The various steps of the proof are detailed in the remaining sections, but we give a high-level outline here.

\begin{proof}[Proof of Theorem~\ref{thm:MainTheorem}]
Using equation \eqref{eq:Ng-S} { with $\mathcal{S}$ defined as in \eqref{eq:S-def}} below we have that
\[
	\mathcal{S}(p,q;s;\phi) = \frac{1- p^{-4s}}{\zeta_{pq}(2s)^2} \scN_g(p,q;s;\phi) - \frac{2\lambda_g(p)}{p^s}\frac{\left(1 - p^{-2s}\right)}{\zeta_{pq}(2s)^2}\scN_g(1,q;s;\phi) + \frac{\scN_g(1,pq;s;\phi)}{\sqrt{p}\zeta_{pq}(2s)^2}.
\]
Specialization of $\zeta_{pq}(2s)^2 \S(p,q;s;\phi)$ at $s = \tfrac12$ yields $\mathcal{M}_g(p,q;\phi)$. Now we apply the reciprocity relation for $\S$ in Theorem \ref{thm:S-reciprocity} below, which gives
\[
	\zeta_{pq}(2s)^2\S(p,q;s,\phi) = \zeta_{pq}(2s)^2 \left(\frac pq\right)^{2s - 1} \frac{\sqrt{p}}{\sqrt{q}} \S(q,p;s;\Phi)
\]
with $\Phi$ as in \eqref{eq:Phi-def}. Specializing to $s = \tfrac12$ yields the result.
\end{proof}

\section{The Bruggeman/Kuznetsov trace formula} \label{sec:kloo-b/k}
 
The purpose of this section is to state the Bruggeman/Kuznetsov trace formula associated to the $0\infty$ cusp pair, which will allow us to relate the moments in Theorem~\ref{thm:MainTheorem} to sums of Kloosterman sums.

Let $\Gamma=\Gamma_0(N)$ and let $\a,\b \in \Q \cup \{\infty\}$ denote two cusps. 
A scaling matrix $\sigma_\a \in \SL_2(\R)$ for the cusp $\a$ satisfies the properties
\begin{equation} \label{eq:ScalingMatrixProperties}
  \sigma_\a \infty = \a, \quad \text{ and } \quad \sigma_\a\inv \Gamma_\a \sigma_\a = \left\{ \pm \pmatrix 1n01 : n \in \Z\right\}
\end{equation}
where $\Gamma_\a$ is the stabilizer of $\a$ in $\Gamma$. 
We are primarily interested in the cases $\a=\infty$ and $\a=0$, for which we have
\begin{equation}
  \sigma_\infty = \pMatrix 1001 \quad \text{ and } \quad \sigma_0 = \pMatrix 0{-1/\sqrt{N}}{\sqrt N}0.
\end{equation}
 
Given a pair of cusps $\a, \b$ and associated scaling matrices $\sigma_\a, \sigma_\b$ 
a set of allowed moduli for the Kloosterman sums is given by
 	\[
 		\mathcal{C}_{\a,\b} = \{\gamma >0: \left(\begin{smallmatrix} *&*\\\gamma&*\end{smallmatrix}\right) \in \sigma_\a\inv \Gamma\sigma_\b\}.
 	\]
The Kloosterman sum for a modulus $\gamma \in \mathcal{C}_{\a,\b}$ is defined as
\begin{equation} \label{eq:Kloo-cusps}
 		S^\Gamma_{\a\b}(m,n;\gamma) = \sum_{\pmatrix ab\gamma d \in \Gamma_\infty \backslash \sigma_\a\inv \Gamma\sigma_\b/\Gamma_\infty} e\left(\frac{am + dn}{\gamma}\right).
\end{equation}
When $\a=\b=\infty$ we obtain the usual Kloosterman sum (for $c\equiv 0\pmod N$)
\begin{equation}
	S^\Gamma_{\infty\infty}(m,n;c) = S(m,n;c) := \sum_{\substack{d\bmod c \\ (d,c)=1}} e\pfrac{\bar dm+d n}{c}.
\end{equation}
For these sums we have Weil's bound
\begin{equation}\label{eq:WeilBound}
	|S(m, n;c)| \leq \tau_0(c) (m,n,c)^{\frac 12} c^{\frac12}
\end{equation}	
where $\tau_0(c)$ is the number of divisors of $c$.
For the $\infty0$ cusp-pair, and for the choices above, the set of allowed moduli is $\mathcal C_{\infty,0}=\{\gamma = c\sqrt{N} : (c,N)=1\}$. 
Then \eqref{eq:Kloo-cusps} can be expressed in terms of $S(m,n;c)$ via
 	\[
 		S_{\infty0}^{\Gamma}(m,n;c\sqrt{N}) = S(\bar N m, n;c).
 	\]

Kloosterman sums appear in the Fourier expansion of the Eisenstein series.
For $n\neq 0$ and $\re (u)>1$ we have (see \cite[Theorem~3.4]{IwaniecSpectralBook})
\begin{equation}
	\rho_{\a\c}(n,u) = \frac{2\pi^u}{\Gamma(u)} |n|^{u-\frac 12} \sum_{c\in \mathcal C_{\c,\a}} \frac{S_{\c\a}(0,n;c)}{c^{2u}}.
\end{equation}
The Kloosterman sum $S(0,n;c)$ is equal to the Ramanujan sum $\sum_{d\mid (n,c)} \mu(c/d) d$ and is multiplicative as a function of $c$.
It follows that, for $q$ prime,
\begin{equation}
	\rho_{\infty 0}(n,u) = \frac{2\pi^u\tau^{(q)}_{u-\frac 12}(n)}{\Gamma(u)\zeta_q(2u)} \quad \text{ and } \quad \rho_{\infty\infty}(n,u) = \frac{2\pi^u}{\Gamma(u)} \left( \frac{\tau_{u-\frac 12}(n)}{\zeta(2u)} - \frac{\tau^{(q)}_{u-\frac 12}(n)}{\zeta_q(2u)} \right),
\end{equation}
where $\tau_w(n) = \sum_{ab=|n|}(a/b)^w$ and $\tau_w^{(q)}(n) = \tau_w(n/(n,q))$ (cf. \cite[(2.27)]{IwaniecSpectralBook}).
From this we can derive a Hecke relation for the pair $(\nu_{\a\c}(n,t),\tau_{it}(p))$, where $p \neq q$ is a prime.
{Suppose that $(n,q)=1$ and that $(\a,\c)=(\infty,\infty)$ or $(\infty,0)$. Then since}
\begin{equation}
	\tau_w^{(q)}(m) \tau_w^{(q)}(n) = \sum_{d\mid (m,n)} \tau_w^{(q)} \pfrac{mn}{d^2}
\end{equation}
{we have} the relation
\begin{equation}
	\nu_{\a\c}(n,t) \tau_{it}(p) = \nu_{\a\c}(np, t) + \delta_{p\mid n} \nu_{\a\c}(n/p,t).
\end{equation}

The Bruggeman/Kuznetsov trace formula relates sums of Kloosterman sums to coefficients of cusp forms and Eisenstein series.
Let $\varphi$ be a smooth function defined on the nonnegative reals satisfying
\begin{equation}\label{eq:phiKuznetsovdGrowth}
	\varphi(0) = 0 \quad \text{ and } \quad \varphi^{(j)}(x) \ll (1+x)^{-2-\epsilon} \quad \text{ for }j=0,1,2,
\end{equation}
and define
\begin{align} \label{eq:KuznetsovTransformDefinitions}
	\varphi_h(\ell) &= \int_0^\infty J_{\ell-1} (x) \varphi(x) \, \frac{dx}{x}, \notag \\
	\varphi_+(t) &= \frac{i}{2\sinh\pi t}\int_{0}^{\infty} \left( J_{2it}(x) - J_{-2it}(x) \right) \varphi(x) \, \frac{dx}{x}.
\end{align}
The following can be found in \cite[Theorem~9.8]{IwaniecSpectralBook}.

\begin{proposition}[Bruggeman/Kuznetsov] \label{prop:0inftyKuznetsov}
Suppose that $m,n\geq 1$.
Let $\varphi$ be a smooth function defined on $[0,\infty)$ which satisfies \eqref{eq:phiKuznetsovdGrowth}. 
Define
\[
	\K(m,n,N;\varphi) = \sum_{(c,N) = 1} \frac{S(\overline{N}m,n;c)}{c\sqrt{N}} \varphi\left(\frac{4\pi \sqrt{mn}}{c\sqrt{N}}\right).
\]
Then 
\[
	\K  = \K^d + \K^c + \K^h,
\]
where
\begin{align}
	\K^d(m,n,N;\varphi) &= \sum_{u_j \in \mathcal U(N)} \varphi_+(t_j) \overline{\nu_{0j}}(m) \nu_{\infty j}(n), \\
	\K^c (m,n,N;\varphi) &= \sums_{\mathfrak c} \frac{1}{4\pi} \int_{-\infty}^\infty \varphi_+(t) \overline{\nu_{0\mathfrak c}}(m,t) \nu_{\infty \mathfrak c}(n,t) \, dt, \\
	\K^h(m,n,N;\varphi) &= \sum_{\ell \equiv 0(2)} i^\ell \varphi_h(\ell) \sum_{f\in \mathcal B_\ell(N)} \overline{\nu_{0f}}(m) \nu_{\infty f}(n).
\end{align}
\end{proposition}

\section{Calculations with the Hecke relations and Kloosterman sums}

Let $\phi$ be as in Theorem~\ref{thm:MainTheorem}.
The purpose of this section is to prove the identity
\begin{equation}\label{eq:Ng-S}
	\scN_g(p,q;s;\phi) = \frac{2\lambda_g(p)}{p^s(1 + p^{-2s})}
	\scN_g(1,q;s;\phi)  - \frac{1}{\sqrt{p}} \frac{\scN_g(1,pq;s;\phi)}{\left(1 - p^{-4s}\right)}   + \zeta_{pq}(2s)^2 \frac{ \S(p,q;s,\phi)}{\left(1 - p^{-4s}\right) },
\end{equation}
where $\S$ is defined in \eqref{eq:S-def} below. 
The following section shows that $\S$ is the quantity that satisfies a natural reciprocity relation.

We begin with explicit computations in the case of the holomorphic spectrum $\scN_g^h(p,q;s;\phi)$.
Suppose that $\sigma>1$.
Opening the $L$-functions {as Dirichlet series} we find that
\begin{equation} \label{eq:N-g-h-1}
	\scN_g^h(p,q;s;\phi) = \zeta_q^2(2s) \sum_{m\geq 1} \frac{\lambda_g(m)}{m^s} \sum_{\ell \equiv 0(2)} i^\ell \phi_h(\ell) \sum_{f\in \H_\ell(q)} \omega_f \nu_f(m) \sum_{n\geq 1} \frac{\lambda_g(n)}{n^s} \nu_f(n) \lambda_f(p).
\end{equation}
We apply the Hecke relations \eqref{eq:hecke} in order to absorb the $\lambda_f(p)$ factor.

\begin{lemma}\label{lem:Lsfg-moment-in-Dirichlet-2}
Let $p$ be a prime number and let $g\in S_\kappa(1)$. 
Suppose that $(\nu,\lambda)$ is a pair of arithmetic functions satisfying the relation
\begin{equation} \label{eq:nu-hecke}
  \nu(n)\lambda(p) = \nu(np) + \delta_{p\mid n} \nu(n/p)
\end{equation}
and the bound $\nu(n) \ll n^{\alpha}$ for some $\alpha>0$.
Then for $\sigma>1+\alpha$ we have
\begin{equation} \label{eq:Mg-S1}
  \sum_{n\geq 1} \frac{\lambda_g(n)}{n^s} \nu(n)\lambda(p)
  = \frac{\lambda_g(p)}{p^s} \sum_{n\geq 1} \frac{\lambda_g(n)}{n^s} \nu(n) + \left(1-p^{-2s}\right) \sum_{n\geq 1} \frac{\lambda_g(n)}{n^s} \nu(np).
\end{equation}
\end{lemma}

\begin{proof}
Using the relation \eqref{eq:nu-hecke}
we find that
\begin{equation}
	\sum_{n\geq 1} \frac{\lambda_g(n)}{n^s} \nu(n)\lambda(p) = \sum_{n\geq 1} \frac{\lambda_g(n)}{n^s} \nu(np) + \sum_{n\geq 1} \frac{\lambda_g(np)}{(np)^s} \nu(n).
\end{equation}
Since $\lambda_g$ also satisfies the Hecke relations \eqref{eq:nu-hecke}
we see that
\begin{equation}
	\sum_{n\geq 1} \frac{\lambda_g(np)}{(np)^s} \nu(n) = \frac{\lambda_g(p)}{p^s} \sum_{n\geq 1} \frac{\lambda_g(n)}{n^s} \nu(n) - \frac{1}{p^{2s}} \sum_{n\geq 1} \frac{\lambda_g(n)}{n^s} \nu(np).
\end{equation}
The lemma follows.
\end{proof}

Applying Lemma~\ref{lem:Lsfg-moment-in-Dirichlet-2} to \eqref{eq:N-g-h-1}, we find that
\begin{equation} \label{eq:N-g-h-2}
	\scN_g^h(p,q;s;\phi) = \frac{\lambda_g(p)}{p^s} \scN_g^h(1,q;s;\phi) + (1-p^{-2s}) \zeta_q(2s)^2 \sum_{m,n\geq 1} \frac{\lambda_g(m)\lambda_g(n)}{(mn)^s} \K^h(m,pn,q;\phi).
\end{equation}
A similar computation confirms that \eqref{eq:N-g-h-2} holds for the discrete and continuous spectra, simply replacing $h$ by $d$ or $c$ above.
After applying Proposition~\ref{prop:0inftyKuznetsov}, it follows that
\begin{equation} \label{N-g-1}
	\scN_g(p,q;s;\phi) = \frac{\lambda_g(p)}{p^s} \scN_g(1,q;s;\phi) + (1-p^{-2s}) \zeta_q(2s)^2 \sum_{m,n\geq 1} \frac{\lambda_g(m)\lambda_g(n)}{(mn)^s} \K(m,pn,q;\phi).
\end{equation}

Our next aim is to relate the sum on the right-hand side of \eqref{N-g-1}
to the sum
\begin{equation} \label{eq:S-def}
  \S(p,q;s,\phi) := \sum_{m,n\geq 1} \frac{\lambda_g(m)\lambda_g(n)}{(mn)^s} \sum_{(c,pq)=1} \frac{S(m\overline{q},np;c)}{c\sqrt{q}} \phi\left(\frac{4\pi \sqrt{mnp}}{c\sqrt{q}}\right).
\end{equation}
This involves a sieving process which leaves us with a sum over $c$ relatively prime to $pq$.

\begin{proposition}\label{prop:Mg-S}
Let $p$ be a prime number and let $g\in S_\kappa(1)$. For $\sigma>\frac 54$ we have
\begin{multline} \label{eq:Mg-S}
  (1+p^{-2s}) \sum_{m,n\geq 1} \frac{\lambda_g(m)\lambda_g(n)}{(mn)^s} \K(m,pn,q;\phi)
  \\ = \S(p,q;s,\phi) - \frac{1}{\sqrt p} \zeta_{pq}(2s)^{-2} \scN_g(1,pq;s;\phi) + \frac{\lambda_g(p)}{p^s} \zeta_q(2s)^{-2} \scN_g(1,q;s;\phi).
\end{multline}
\end{proposition}

In the proof of Proposition~\ref{prop:Mg-S} we will need the following facts about Kloosterman sums, all of which follow from standard exponential sum manipulations.
\begin{lemma} \label{lem:kloo-props}
Suppose that $m,n,c \in \Z$ with $c>0$ and that $p$ is prime.
\begin{enumerate}
  \item If $p\mid\mid c$ then
  \begin{equation} \label{eq:kloo-p-1}
    S(m,pn;c) = S(\bar p m, n; c/p) \times
    \begin{cases}
      -1 & \text{ if } p\nmid m, \\
      p-1 & \text{ if }p\mid m,
    \end{cases}
  \end{equation}
  where $\bar pp\equiv 1\pmod{c/p}$.
  \item If $p^2 \mid c$ then 
  \begin{equation} \label{eq:kloo-p-2}
    S(m,pn;c)=0 \quad \text{ unless } p\mid m.
  \end{equation}
  \item We have
  \begin{equation} \label{eq:kloo-p-3}
    S(pm,pn,p^2c) = p S(m,n,pc).
  \end{equation}
\end{enumerate}
\end{lemma}

\begin{proof}[Proof of Proposition~\ref{prop:Mg-S}]
Let  
\begin{align}
	\scR :=& \sum_{m,n\geq 1} \frac{\lambda_g(m)\lambda_g(n)}{(mn)^s} \K(m,pn,q;\phi) \\
	=& \sum_{m,n\geq 1} \frac{\lambda_g(m)\lambda_g(n)}{(mn)^s} \sum_{(c,q)=1} \frac{S(\bar qm,pn;c)}{c\sqrt q} \phi \pfrac{4\pi\sqrt{mnp}}{c\sqrt q},
\end{align}
and write $\scR = \S + \scT$,
where $\S=\S(p,q;s,\phi)$ is defined in \eqref{eq:S-def} and $\scT$ comprises those terms of $\scR$ with $p\mid c$.
If furthermore $p^2 \mid c$, we apply \eqref{eq:kloo-p-2}, and if $p\mid\mid c$, we apply \eqref{eq:kloo-p-1}, obtaining
\begin{multline}
	\scT = 
	\sum_{\substack{m,n\geq 1}} \frac{\lambda_g(pm)\lambda_g(n)}{(pm)^s n^s} \sum_{\substack{(c,q)=1 \\ p\mid c}} \frac{S(\bar qpm,pn;pc)}{pc\sqrt q} \phi \pfrac{4\pi \sqrt{mn}}{c\sqrt{q}} \\
	- 	\sum_{\substack{m,n\geq 1\\p\nmid m}} \frac{\lambda_g(m)\lambda_g(n)}{m^s n^s} \sum_{\substack{(c,pq)=1}} \frac{S(\bar {qp}m,n;c)}{pc\sqrt q} \phi \pfrac{4\pi \sqrt{mn}}{c\sqrt{pq}} \\
	+(p-1) \sum_{\substack{m,n\geq 1\\p\mid m}} \frac{\lambda_g(m)\lambda_g(n)}{m^s n^s} \sum_{\substack{(c,pq)=1}} \frac{S(\bar {qp}m,n;c)}{pc\sqrt q} \phi \pfrac{4\pi \sqrt{mn}}{c\sqrt{pq}}. 
\end{multline}
Now we apply \eqref{eq:kloo-p-3} to the terms of the first sum.
In the last line we separate $p-1$, combining the $p$ term with the first line, and the $-1$ term with the second. We find that
\begin{multline}
	\scT = \sum_{\substack{m,n\geq 1}} \frac{\lambda_g(pm)\lambda_g(n)}{(pm)^s n^s} \sum_{\substack{(c,q)=1}} \frac{S(\bar qm,n;c)}{c\sqrt q} \phi \pfrac{4\pi\sqrt{mn}}{c\sqrt{q}} \\
	- \frac{1}{\sqrt p}	\sum_{\substack{m,n\geq 1}} \frac{\lambda_g(m)\lambda_g(n)}{m^s n^s} \sum_{\substack{(c,pq)=1}} \frac{S(\bar {qp}m,n;c)}{c\sqrt {pq}} \phi \pfrac{4\pi\sqrt{mn}}{c\sqrt{pq}}.
\end{multline}
Finally, we apply the Hecke relation again to $\lambda_g(pm)$ in the first term, and conclude that
\begin{equation}
	\scT = \frac{\lambda_g(p)}{p^s} \zeta_q(2s)^{-2} \scN_g(1,q;s;\phi) - \frac{1}{p^{2s}} \scR - \frac{1}{\sqrt p} \zeta_{pq}(2s)^{-2} \scN_g(1,pq;s;\phi)
\end{equation}
since
\begin{equation}
	S(\bar q pm,n;c) = S(pm,\bar q n;c) = S(\bar q n, pm;c).
\end{equation}
On the other hand, $\scT = \scR - \S$, so we have
\begin{equation}
	(1+p^{-2s})\scR = \S - \frac{1}{\sqrt p} \zeta_{pq}(2s)^{-2} \scN_g(1,pq;s;\phi) + \frac{\lambda_g(p)}{p^s} \zeta_q(2s)^{-2} \scN_g(1,q;s;\phi),
\end{equation}
as desired.
\end{proof}

Now let us combine the last two results. Combining \eqref{N-g-1} and Proposition \ref{prop:Mg-S} we obtain
\begin{multline}
	\scN_g(p,q;s;\phi) = \frac{\lambda_g(p)}{p^s}\scN(1,q,s,\phi) + \frac{\zeta_{pq}(2s)^2 \left(1 +   p^{-2s}\right)}{\left(1 - p^{-2s}\right)\left(1 + p^{-2s}\right)} \sum_{m,n\geq 1} \frac{\lambda_g(m)\lambda_g(n) }{(mn)^s} K(m,pn;s;\phi)\\
	= \left(\frac{\lambda_g(p)}{p^s} + \frac{\lambda_g(p)}{p^s}\frac{\left(1 - p^{-2s}\right) }{\left(1 + p^{-2s}\right)}\right) \scN_g(1,q;s;\phi)  - \frac{1}{\sqrt{p}} \frac{\scN_g(1,pq;s;\phi)}{\left(1 - p^{-4s}\right)}   + \zeta_{pq}(2s)^2 \frac{ S(p,q;s,\phi)}{\left(1 - p^{-4s}\right)} .
\end{multline}
Rearanging the terms in the first parentheses we obtain \eqref{eq:Ng-S}.

\section{Reciprocity for \texorpdfstring{$\S(p,q;s,\phi)$}{S(p,q;s,phi)}}

In this section we will prove the following reciprocity relation for $\S(p,q;s,\phi)$.
\begin{theorem} \label{thm:S-reciprocity}
Let $\phi$ be a smooth test function satisfying $\phi^{(j)}(0)  = 0$ and $\phi^{(j)}(x) \ll (1 + x)^{-A}$ for $0\leq j\leq 12$ and for some $A>12$. 
Suppose that $\Re(s) = \sigma>\frac 54$ and that $p,q$ are distinct primes. 
Then
\begin{equation}
	\sqrt q \, \S(p,q;s,\phi) = \pfrac{p}{q}^{2s-1} \sqrt p \, \S(q,p;s,\Phi),
\end{equation}
where, for any $\xi$ satisfying $0<\xi+12<A$, $\Phi(x)$ is defined as
\begin{equation} \label{eq:Phi-def}
	\Phi(x) = \Phi_{\kappa,s}(x) :=  \frac{1}{2\pi i}  \int_{(\xi)} \tilde\phi(u) 2^{-u} \left[\frac{\Gamma(\frac{\kappa+1}{2}-s-\frac{u}{2})}{\Gamma(\frac{\kappa-1}{2}+s+\frac{u}{2})}\right]^2 \left(\frac{x}{2}\right)^{u+4s-2} \, du
\end{equation}
and $\tilde\phi$ is the Mellin transform
\begin{equation}
  \tilde\phi(u) = \int_0^\infty t^u \phi(t) \, \mfrac{dt}{t}.
\end{equation}
Furthermore, the function $\Phi_{\kappa, s}$ can be analytically continued to $\frac12 \leq \sigma \leq \frac54 + \epsilon$ and it satisfies \eqref{eq:phiKuznetsovdGrowth} in that region.
\end{theorem}

Suppose that $\phi^{(j)}(0)=0$ and $\phi^{(j)}(x) \ll (1+x)^{-A}$ for all $j\in \{0,1,2,\ldots, J\}$, {for some integer $J$}. 
{We will later see that we may take $J = 12$.}
By the decay of $\phi$ and its derivatives at $0$ and $\infty$, we may apply integration by parts $j$ times and obtain that
\begin{equation}\label{eq:MellinPhiBound}
	\widetilde{\phi}(u) = (-1)^j \int_{0}^\infty \phi^{(j)}(x) \mfrac{x^{u + j}}{u (u + 1) \cdots (u + j-1)}\, \mfrac{d x}{x} \ll (1 + |u|)^{-j}
\end{equation} 
as long as the integral converges.
Near $0$ we have $\phi^{(j)}(x)\ll x^{J-j}$, so the integrand is majorized by $x^{J+\xi-1}$  {(with $\xi = \Re(u)$)} as $x\to 0$.
For large $x$ the integrand is majorized by $x^{-A+\xi+j-1}$.
It follows that the integral is absolutely convergent (and thus $\eqref{eq:MellinPhiBound}$ holds with $j=J$) as long as $-J<\xi<A-J$.

Starting with \eqref{eq:S-def}, we write $\phi$ as the inverse Mellin transform of $\widetilde{\phi}$ via
\begin{equation}
	\phi(x) = \frac{1}{2\pi i} \int_{(\xi)} \tilde\phi(u) x^{-u} \, du.
\end{equation}
We then interchange integral and summation to obtain
\begin{align}\label{eq:S-integral-series}
	\sqrt q \, \S(p,q;s,\phi) &= \sum_{m,n\geq 1} \frac{\lambda_g(m)\lambda_g(n)}{(mn)^s} \sum_{(c,pq)=1} \frac{S(m\overline{q},np;c)}{c} \frac{1}{2\pi i} 
	\int_{(\xi)} \widetilde{\phi}(u)\left(\frac{4\pi \sqrt{mnp}}{c\sqrt{q}}\right)^{-u} \, du \notag\\
	&= \frac{1}{2\pi i} \int_{(\xi)} \tilde\phi(u) \pfrac{4\pi \sqrt p}{\sqrt q}^{-u} \sum_{(c,pq)=1} \sum_{m,n\geq 1} \frac{\lambda_g(m)\lambda_g(n)S(\bar qm, pn; c)}{c^{1-u}(mn)^{s+\frac u2}} \, du.
\end{align}
As the following proposition shows, the innermost sum above satisfies a functional equation which exchanges the roles of $p$ and $q$.

\begin{proposition} \label{prop:D-props}
Let $g\in S_k(1)$ be a newform.
% \blue{(\emph{we need it to be a hecke form, don't we?, even if not we do write $\lambda_g(n)$, is it worth the trouble? Anyone who wants can take a linear combination of the results at the end if they need it})}
Suppose that $a$, $b$, and $c$ are integers with $c$ positive and $(a,c)=(b,c)=1$. 
For $\sigma>1$, define
\begin{equation} \label{eq:H-def}
	D_g(a,b,c;s) := \sum_{m,n\geq 1} \frac{\lambda_g(m)\lambda_g(n)S(am, bn; c)}{(mn)^s}.
\end{equation}
Then $D_g(a,b,c;s)$ extends to an entire function of $s$ which satisfies the functional equation
\begin{equation} \label{eq:H-func-eq}
	D_g(a,b,c;s) = \pfrac{2\pi}{c}^{4s-2} \left[\frac{\Gamma(\frac{\kappa+1}{2}-s)}{\Gamma(\frac{\kappa-1}{2}+s)}\right]^2 D_g(\bar a, \bar b, c;1-s),
\end{equation}
where $a\bar a \equiv b\bar b \equiv 1 \pmod{c}$.
Furthermore, for any $\epsilon>0$, as $|t|\to\infty$ with $\sigma$ bounded we have
\begin{equation} \label{eq:H-bound}
	D_g(a,b,c;s) \ll 
	\begin{cases}
		c^{\frac 12+\epsilon} & \text{ if } \sigma \geq 1+\epsilon, \\
		c^{\frac 52-2\sigma+\epsilon}|t|^{2-2\sigma+\epsilon} & \text{ if } -\epsilon \leq \sigma \leq 1+\epsilon, \\
		c^{\frac 52-4\sigma+\epsilon} |t|^{2-4\sigma} & \text{ if } \sigma\leq -\epsilon.
	\end{cases}
\end{equation}
\end{proposition}

\begin{proof}
Suppose that $\sigma>1$.
Opening up the Kloosterman sum, we find that
\begin{equation}\label{eq:Dg-in-L}
	D_g(a,b,c;s) = \sums_{d\bmod c} L(s,g,\tfrac{a\bar d}{c}) L(s,g,\tfrac{b\bar d}{c}),
\end{equation}
where
\begin{equation}\label{eq:LgAdditiveTwist}
	L(s,g,x) := \sum_{m=1}^\infty \frac{\lambda_g(m)e(mx)}{m^s}.
\end{equation}
The $L$-series $L(s,g,x)$ extends to an entire function of $s$ which satisfies, for $x=\frac dc$ with $(d,c)=1$,  the functional equation
\begin{equation}\label{eq:LgAdditiveTwistFE}
	L(s,g,\tfrac dc) = (-1)^{\frac \kappa2} \pfrac{2\pi}{c}^{2s-1} \frac{\Gamma(\frac{\kappa+1}2-s)}{\Gamma(\frac{\kappa-1}2+s)} L(1-s,g,-\tfrac{\bar d}{c}).
\end{equation}
Thus $D_g(a,b,c;s)$ inherits the functional equation \eqref{eq:H-func-eq} after using that $S(-a,-b;c)=S(a,b;c)$.
To prove the estimates \eqref{eq:H-bound} we apply the Phragmen-Lindel\"of principle (see, e.g. IK, Theorem 5.53).
The estimate in the range $\sigma \geq 1+\epsilon$ follows from a standard computation involving the Weil bound \eqref{eq:WeilBound}.
In the range $\sigma \leq -\epsilon$ we apply the functional equation \eqref{eq:H-func-eq}, together with Stirling's formula and the bound in the $\sigma\geq 1+\epsilon$ range.
Then Phragmen-Lindel\"of yields the bound in the middle range $-\epsilon\leq \sigma \leq 1+\epsilon$.
\end{proof}

The innermost sum in \eqref{eq:S-integral-series} equals $c^{u-1}D_g(\overline{q}, p, c;s + \tfrac{u}{2})$.
We would like to apply the functional equation \eqref{eq:H-func-eq} and write the result as a sum of Kloosterman sums. 
To do this, we will need to work in the region of absolute convergence of the function $D_g(q,\bar p,c;1-\frac u2-s)$, so we move the line of integration in $u$ to $\xi=-2\sigma-\epsilon$.
By the estimates \eqref{eq:H-bound} and \eqref{eq:MellinPhiBound}, the resulting integral converges as long as $J\geq 4$. {Note that $\sigma > \tfrac 54$ and hence the sum over $c$ converges.}
The function is entire, as can be seen from \eqref{eq:Dg-in-L} and we do not pick up any poles in the process\footnote{Note that if we were to replace $\lambda_g(n)$ with $\tau_w(n)$, the resulting function $D$ would be a finite sum of products of two Estermann zeta functions, and we would have picked up poles which in turn would constitute the main terms one encounters in the reciprocity formula of the fourth moment.}.

\begin{proof}[Proof of Theorem~\ref{thm:S-reciprocity}]
Starting with the \eqref{eq:S-integral-series} represenation of $\sqrt q \, \S(p,q;s)$ with $\xi$ satisfying $\xi = - 2\sigma-\epsilon$, we apply the functional equation \eqref{eq:H-func-eq} to obtain
\begin{multline}
	\sqrt {q/p} \, \S(p,q;s;\phi)
	= \left(\frac{p}{q}\right)^{2s-1} \frac{1}{2\pi i}  \int_{(\xi)} \widetilde{\phi}(u)  \left[\frac{\Gamma(\frac{\kappa+1}{2}-s-\frac{u}{2})}{\Gamma(\frac{\kappa-1}{2}+s+\frac{u}{2})}\right]^2 2^{-4s - 2u + 2}\\
	\times  \sum_{m,n\geq 1} \frac{\lambda_g(m)\lambda_g(n)}{(mn)^{s}} \sum_{(c,pq)=1} \frac{S(qm,\bar pn;c)}{c\sqrt p} \pfrac{4\pi \sqrt{mnq}}{c\sqrt p}^{4s + u - 2} \, du.
\end{multline}
The integral and sums are absolutely convergent for $\sigma>\frac 54$, so we can interchange them to obtain
\begin{align}
	\sqrt {q/p} \, \S(p,q;s;\phi) &= \left(\frac{p}{q}\right)^{2s-1} \sum_{m,n\geq 1} \frac{\lambda_g(m)\lambda_g(n)}{(mn)^s} \sum_{(c,pq) = 1} \frac{S(qm, \overline{p}n;c)}{c\sqrt{p}} \Phi\left(\frac{4\pi \sqrt{mnq}}{c\sqrt{p}}\right).
	\end{align}
The sums on the right hand side are exactly $\S(q,p;s;\Phi)$.

Let us finish the proof by specifying the conditions for $\phi$ so that $\Phi$ satisfies  \eqref{eq:phiKuznetsovdGrowth}. 
This is necessary in order for us to apply the Kuznetsov formula to the sum of Kloosterman sums with the weight function $\Phi$. 

So far we have required that $J\geq 4$ and that $-J<\xi<A-J$.
By Stirling's formula the integrand in the definition of $\Phi$ is majorized by $( 1 + |\Im(u)|)^{-J +  2 - 4\sigma - 2\xi}$, so the integral converges as long as 
$2\xi+4\sigma > 3 - J$.
Suppose that $\sigma \geq \frac 12$.
Then $\Phi(0)=0$ as long as we can take $\xi>0$ (note that we can easily stay  away from the poles of the gamma factor since $\kappa\geq 12$).
This requires $A>J$.

On the other hand, to estimate $\Phi^{(\ell)}(x)$ for large $x$ and for $\ell=0,1,2$, we want to move $\xi$ as far to the left as possible. 
Differentiating under the integral, we find that
\begin{align}
	\Phi^{(\ell)} (x) \ll \frac{1}{2\pi i} \int_{(\xi)}|\tilde\phi(u)| \left|\frac{\Gamma(\frac{\kappa+1}{2}-s-\frac{u}{2})}{\Gamma(\frac{\kappa-1}{2}+s+\frac{u}{2})}\right|^2 x^{4\sigma-2+\xi- \ell} ( 1 + | u|)^{\ell} \, du \ll (1 + x)^{4\sigma - 2  + \xi - \ell},
\end{align}
assuming the integral converges.
To ensure convergence, we need that $2\xi+4\sigma>\ell+3-J$, and for $\Phi^{(\ell)} (1+x) \ll x^{-2-\epsilon}$ to hold we need that $4\sigma  + \xi <\ell$.
Suppose that $\frac 12\leq \sigma\leq \frac 54+\epsilon$. 
Then we need
\begin{equation}
	\ell-J+1 < 2\xi < 2\ell-10 \qquad \text{ for }\ell=0,1,2.
\end{equation}
For this interval to be nonempty, it is enough to take $J\geq 12$.
\end{proof}

\bibliographystyle{plain}
\bibliography{../biblio}

\end{document}